\documentclass[12pt,reqno,draft]{amsart}
\usepackage{amsmath,amsthm,amssymb,amscd,url}

\begin{document}

\newcommand{\TITLE}{Terms in elliptic divisibility sequences
  divisible by their indices}
\newcommand{\TITLERUNNING}{Index divisibility in elliptic
  divisibility sequences}
\newcommand{\DATE}{\today}
\newcommand{\VERSION}{7}

\theoremstyle{plain} 
\newtheorem{theorem}{Theorem} 
\newtheorem{conjecture}[theorem]{Conjecture}
\newtheorem{proposition}[theorem]{Proposition}
\newtheorem{lemma}[theorem]{Lemma}
\newtheorem{corollary}[theorem]{Corollary}

\theoremstyle{definition}
\newtheorem*{definition}{Definition}

\theoremstyle{remark}
\newtheorem{remark}[theorem]{Remark}
\newtheorem{example}[theorem]{Example}
\newtheorem{question}[theorem]{Question}
\newtheorem*{acknowledgement}{Acknowledgements}

\def\BigStrut{\vphantom{$(^{(^(}_{(}$}} 

\newenvironment{notation}[0]{%
  \begin{list}%
    {}%
    {\setlength{\itemindent}{0pt}
     \setlength{\labelwidth}{4\parindent}
     \setlength{\labelsep}{\parindent}
     \setlength{\leftmargin}{5\parindent}
     \setlength{\itemsep}{0pt}
     }%
   }%
  {\end{list}}

\newenvironment{parts}[0]{%
  \begin{list}{}%
    {\setlength{\itemindent}{0pt}
     \setlength{\labelwidth}{1.5\parindent}
     \setlength{\labelsep}{.5\parindent}
     \setlength{\leftmargin}{2\parindent}
     \setlength{\itemsep}{0pt}
     }%
   }%
  {\end{list}}
\newcommand{\Part}[1]{\item[\upshape#1]}

%
\newcommand{\EndProofAtDisplay}{\renewcommand{\qedsymbol}{}}
\newcommand{\qedtag}{\tag*{\qedsymbol}}

\renewcommand{\a}{\alpha}
\renewcommand{\b}{\beta}
\newcommand{\g}{\gamma}
\renewcommand{\d}{\delta}
\newcommand{\e}{\epsilon}
\newcommand{\f}{\phi}
\newcommand{\fhat}{{\hat\phi}}
\renewcommand{\l}{\lambda}
\renewcommand{\k}{\kappa}
\newcommand{\lhat}{\hat\lambda}
\newcommand{\m}{\mu}
\renewcommand{\o}{\omega}
\renewcommand{\r}{\rho}
\newcommand{\rbar}{{\bar\rho}}
\newcommand{\s}{\sigma}
\newcommand{\sbar}{{\bar\sigma}}
\renewcommand{\t}{\tau}
\newcommand{\z}{\zeta}

\newcommand{\D}{\Delta}
\newcommand{\F}{\Phi}
\newcommand{\G}{\Gamma}

\newcommand{\ga}{{\mathfrak{a}}}
\newcommand{\gb}{{\mathfrak{b}}}
\newcommand{\gc}{{\mathfrak{c}}}
\newcommand{\gd}{{\mathfrak{d}}}
\newcommand{\gm}{{\mathfrak{m}}}
\newcommand{\gn}{{\mathfrak{n}}}
\newcommand{\gp}{{\mathfrak{p}}}
\newcommand{\gq}{{\mathfrak{q}}}
\newcommand{\gP}{{\mathfrak{P}}}
\newcommand{\gQ}{{\mathfrak{Q}}}

\def\Acal{{\mathcal A}}
\def\Bcal{{\mathcal B}}
\def\Ccal{{\mathcal C}}
\def\Dcal{{\mathcal D}}
\def\Ecal{{\mathcal E}}
\def\Fcal{{\mathcal F}}
\def\Gcal{{\mathcal G}}
\def\Hcal{{\mathcal H}}
\def\Ical{{\mathcal I}}
\def\Jcal{{\mathcal J}}
\def\Kcal{{\mathcal K}}
\def\Lcal{{\mathcal L}}
\def\Mcal{{\mathcal M}}
\def\Ncal{{\mathcal N}}
\def\Ocal{{\mathcal O}}
\def\Pcal{{\mathcal P}}
\def\Qcal{{\mathcal Q}}
\def\Rcal{{\mathcal R}}
\def\Scal{{\mathcal S}}
\def\Tcal{{\mathcal T}}
\def\Ucal{{\mathcal U}}
\def\Vcal{{\mathcal V}}
\def\Wcal{{\mathcal W}}
\def\Xcal{{\mathcal X}}
\def\Ycal{{\mathcal Y}}
\def\Zcal{{\mathcal Z}}

\renewcommand{\AA}{\mathbb{A}}
\newcommand{\BB}{\mathbb{B}}
\newcommand{\CC}{\mathbb{C}}
\newcommand{\FF}{\mathbb{F}}
\newcommand{\GG}{\mathbb{G}}
\newcommand{\NN}{\mathbb{N}}
\newcommand{\PP}{\mathbb{P}}
\newcommand{\QQ}{\mathbb{Q}}
\newcommand{\RR}{\mathbb{R}}
\newcommand{\ZZ}{\mathbb{Z}}

\def \bfa{{\mathbf a}}
\def \bfb{{\mathbf b}}
\def \bfc{{\mathbf c}}
\def \bfe{{\mathbf e}}
\def \bff{{\mathbf f}}
\def \bfF{{\mathbf F}}
\def \bfg{{\mathbf g}}
\def \bfn{{\mathbf n}}
\def \bfp{{\mathbf p}}
\def \bfr{{\mathbf r}}
\def \bfs{{\mathbf s}}
\def \bft{{\mathbf t}}
\def \bfu{{\mathbf u}}
\def \bfv{{\mathbf v}}
\def \bfw{{\mathbf w}}
\def \bfx{{\mathbf x}}
\def \bfy{{\mathbf y}}
\def \bfz{{\mathbf z}}
\def \bfX{{\mathbf X}}
\def \bfU{{\mathbf U}}
\def \bfmu{{\boldsymbol\mu}}

\newcommand{\Gbar}{{\bar G}}
\newcommand{\Kbar}{{\bar K}}
\newcommand{\kbar}{{\bar k}}
\newcommand{\Obar}{{\bar O}}
\newcommand{\Pbar}{{\bar P}}
\newcommand{\Rbar}{{\bar R}}
\newcommand{\Qbar}{{\bar Q}}
\newcommand{\QQbar}{{\bar{\QQ}}}


\newcommand{\Anom}{\mathcal{A}_1}
\newcommand{\Aliq}{\mathcal{A}}
\newcommand{\AliqGen}{\mathcal{A}_{\textup{gen}}}
\newcommand{\Aseq}{\mathsf{A}}
\newcommand{\Arrow}{\operatorname{Arr}}
\newcommand{\Aut}{\operatorname{Aut}}
\newcommand{\defeq}{\stackrel{\textup{def}}{=}}
\newcommand{\Disc}{\operatorname{Disc}}
\renewcommand{\div}{\operatorname{div}}
\newcommand{\Div}{\operatorname{Div}}
\newcommand{\Etilde}{{\tilde E}}
\newcommand{\End}{\operatorname{End}}
\newcommand{\EDS}{\mathsf{D}}
\newcommand{\WardEDS}{\mathsf{W}}
\newcommand{\Fix}{\operatorname{Fix}}
\newcommand{\Frob}{\operatorname{Frob}}
\newcommand{\Gal}{\operatorname{Gal}}
\newcommand{\GCD}{{\operatorname{GCD}}}
\renewcommand{\gcd}{{\operatorname{gcd}}}
\newcommand{\GL}{\operatorname{GL}}
\newcommand{\hhat}{{\hat h}}
\newcommand{\Hom}{\operatorname{Hom}}
\newcommand{\Ideal}{\operatorname{Ideal}}
\newcommand{\Image}{\operatorname{Image}}
\newcommand{\In}{\operatorname{InDeg}}
\newcommand{\IR}[1]{{\textup{I}$_{#1}$}}
\newcommand{\longhookrightarrow}{\lhook\joinrel\relbar\joinrel\rightarrow}
\newcommand{\LS}[2]{\genfrac(){}{}{#1}{#2}}  
\newcommand{\Lucas}{\mathsf{L}}
\newcommand{\MOD}[1]{~(\textup{mod}~#1)}
\renewcommand{\pmod}{\MOD}
\newcommand{\Neron}{{\mathfrak E}} 
\newcommand{\Norm}{{\textup{\textsf{N}}}}
\newcommand{\NS}{\operatorname{NS}}
\newcommand{\ns}{{\textup{ns}}}
\newcommand{\notdivide}{\nmid}
\newcommand{\longonto}{\longrightarrow\hspace{-10pt}\rightarrow}
\newcommand{\onto}{\rightarrow\hspace{-10pt}\rightarrow}
\newcommand{\ord}{\operatorname{ord}}
\newcommand{\Out}{\operatorname{OutDeg}}
\newcommand{\Parity}{\operatorname{Parity}}
\newcommand{\Pic}{\operatorname{Pic}}
\newcommand{\Prob}{\operatorname{Prob}}
\newcommand{\Proj}{\operatorname{Proj}}
\newcommand{\rank}{\operatorname{rank}}
\newcommand{\res}{\operatornamewithlimits{res}}
\newcommand{\Resultant}{\operatorname{Resultant}}
\renewcommand{\setminus}{\smallsetminus}
\newcommand{\sign}{\operatorname{Sign}}
\newcommand{\Spec}{\operatorname{Spec}}
\newcommand{\Support}{\operatorname{Support}}
\newcommand{\tors}{{\textup{tors}}}
\newcommand{\Tr}{\operatorname{Tr}}
\newcommand{\Vertex}{\operatorname{Ver}}
\newcommand\W{W^{\vphantom{1}}}
\newcommand{\Wtilde}{{\widetilde W}}
\newcommand{\<}{\langle}
\renewcommand{\>}{\rangle}
\newcommand{\semiquad}{\hspace{.5em}}

\hyphenation{para-me-tri-za-tion}


\title[\TITLERUNNING]{\TITLE}
\date{\DATE, Draft \#\VERSION}
\author{Joseph H. Silverman}
\address{Mathematics Department, Box 1917, Brown University, 
Providence, RI 02912 USA} 
\email{jhs@math.brown.edu} 
\author{Katherine E. Stange} 
\address{%
Department of Mathematics, Simon Fraser University,
8888 University Drive, Burnaby, BC, Canada V5A 1S6,
and 
Pacific Institute for the Mathematical Sciences,
200 1933 West Mall, Vancouver, BC, Canada V6T 1Z2}
\email{stange@pims.math.ca}

\thanks{The first author's research supported by 
NSF DMS-0650017 and DMS-0854755.
The second author's research supported by NSERC PDF-373333.
}


\begin{abstract}
Let $\EDS=(D_n)_{n\ge1}$ be an elliptic divisibility sequence. We
study the set~$\Scal(\EDS)$ of indices~$n$ satisfying $n\mid D_n$.  In
particular, given an index~$n\in\Scal(\EDS)$, we explain how to
construct elements~$nd\in\Scal(\EDS)$, where~$d$ is either a prime
divisor of~$D_n$, or~$d$ is the product of the primes in an aliquot
cycle for~$\EDS$. We also give bounds for the exceptional indices that
are not constructed in this way.
\end{abstract}

\maketitle


\section*{Introduction}
\label{section:introduction}

In this note we investigate the terms in elliptic divisibility
sequences that are divisible by their indices. The analogous problem
has been studied for a number of other types of sequences.  
For example, the Fibonacci sequence $(F_n)_{n\ge1}$ satisfies
\[
  n\mid F_n \quad\Longleftrightarrow\quad
  n\in\{1, 5, 12, 24, 25, 36, 48, 60, 72, 96,\ldots\}.
\]
See
\cite{MR1131414,MR0349567,MR0130205,Smyth,MR1271392,MR1393479,walsh}
for results on index divisibility in the Fibonacci sequence and in more
general Lucas sequences.  To cite another example, values of~$n$ that
divide~$a^n-a$ are called \emph{pseudoprimes to the base~$a$}. They
have been studied for their intrinsic interest and for applications to
cryptography~\cite{MR2173389,MR2463535,Pomerance,MR1737681,MR674829}.
\par 
In general, for any integer sequence~$\Aseq=(A_n)_{n\ge1}$ we
define the \emph{index divisibility set} of~$\Aseq$ to be
\[
  \Scal(\Aseq) = \bigl\{ n\ge 1 : n\mid A_n\bigr\}.
\]
Our goal is to build~$\Scal(\Aseq)$ multiplicatively via a directed
graph that connects each element~$n\in\Scal(\Aseq)$ to its (minimal)
multiples in~$\Scal(\Aseq)$. Thus we define a directed graph by taking
the set~$\Scal(\Aseq)$ to be the set of vertices and by drawing an
arrow from~$n$ to~$m$ if the following two conditions are true:
\begin{parts}
\Part{(1)}
$n\mid m$.
\Part{(2)}
If $k\in\Scal(\Aseq)$ satisfies $n\mid k\mid m$, then  $k=n$ or $k=m$.
\end{parts}
In other words, if we partially order~$\Scal(\Aseq)$ by divisibility,
then we draw an arrow from~$n$ to~$m$ if $n$ is strictly smaller
than~$m$ and if there are no elements of~$\Scal(\Aseq)$ that are
strictly between~$n$ and~$m$.
\par
We denote the set of arrows by~$\Arrow(\Aseq)$, and we assign
weight~$m/n$ to the arrow $(n\to m)$.  (Smyth~\cite[Section~8]{Smyth}
defines a similar structure, but he allows only arrows of prime
weight, so his graphs may be disconnected.)

\begin{definition}
Let~$E/\QQ$ be an elliptic curve given by a Weierstrass equation and
let~$P\in E(\QQ)$ be a nontorsion point.  The \emph{elliptic
  divisibility sequence} (EDS) associated to the pair~$(E,P)$ is the
sequence of positive integers~$\EDS=(D_n)_{n\ge1}$ obtained by writing
\[
  x\bigl([n]P\bigr) = \frac{A_n}{D_n^2} \in \QQ
\]
as a fraction in lowest terms.  The EDS is \emph{minimal} if~$E$ is
given by a minimal Weierstrass equation.
An EDS is \emph{normalized} if $D_1=1$.  An arbitrary EDS~$(D_n)_{n\ge1}$ can
be normalized by a change of variables in the defining Weierstrass
equation, in which case the new EDS is~$(D_n/D_1)_{n\ge1}$. Note, however,
that the normalized sequence may not be minimal.
\end{definition}

We remark that there is an alternative definition of EDS via a
non-linear recurrence that gives almost the same set of sequences; see
Remark~\ref{remark:wardedsdef} for further details. We also note
that, as its name suggests, an~EDS is a divisibility sequence, i.e.,
\[
  m\mid n \implies D_m\mid D_n.
\]

The arithmetic properties of EDS have been extensively studied as
examples of nontrivial nonlinear recursions that possess enough
additional structure to make them amenable to Diophantine analysis.
See for example Ward's original papers~\cite{MR0027286,MR0023275},
subsequent work
including~\cite{MR1815962,MR2164113,MR2220263,MR2226354}, and
applications of~EDS to Hilbert's 10th problem and to
cryptography~\cite{MR2377127,MR2480276,MR1992832,MR2423649}.

Although EDS are defined via a non-linear process, their underlying
structure comes from the associated elliptic curve. They are thus a
natural generalization of linear recursions such as the Fibonacci and
Lucas sequences, which are associated to the multiplicative group.

\begin{example}
\label{example:N37curve}
Let~$\EDS$ be the EDS
\begin{multline*}
  \EDS = 
  (1, 1, 1, 1, 2, 1, 3, 5, 7, 4, 23, 29, 59, 129,\\
      314, 65, 1529, 3689, 8209, 16264, 83313,\dots)
\end{multline*}
associated to the elliptic curve and point
\[
  E:y^2+y=x^3-x,\qquad P=(0,0).
\]
Then
\[
  \Scal(\EDS)
   = \{ 1, 40, 53, 63, 80, 127, 160, 189, 200, 320, 400, 441, 443,\dots\}.
\]
We remark that the sequence~$\EDS$ grows very rapidly. Thus the first
two nontrivial elements of~$\Scal(\EDS)$ in this example come from
\begin{align*}
  D_{40} &= 40\cdot 13526278251270010,\\
  D_{53}&=53\cdot 299741133691576877400370757471.
\end{align*}
The reader may have noticed that~$\Scal(\EDS)$ contains the
primes~$53$,~$127$, and~$443$, which are the first three anomalous
primes for~$E$, i.e., primes satisfying $\#E(\FF_p)=p$. This is
not a coincidence.
\end{example}

Smyth has given an explicit description of index divisibility for
Lucas sequences. For comparison with our results, we state one of his
theorems, reformulated using the terminology of directed graphs.

\begin{theorem}
\label{theorem:smyth}
\textup{(Smyth \cite[Theorem 1]{Smyth})}
Let~$a,b\in\ZZ$, and let $\Lucas = (L_n)_{n\ge1}$ be the associated
Lucas sequence of the first kind, i.e., defined by the recursion
\[
  L_{n+2} = aL_{n+1} - bL_n,\qquad L_0=0,\quad L_1=1.
\]
Let $\D=a^2-4b$. Then the arrows originating at a vertex
$n\in\Scal(\Lucas)$ are
\[
  \{ n\to np : \text{$p$ is prime and $p\mid L_n\D$} \}
    \cup \Bcal_{a,b},
\]
where
\[
  \Bcal_{a,b} = \begin{cases}
    \{1\to6\}
         &\text{if $a\equiv3\pmod6$ and $b\equiv\pm1\pmod6$,} \\
    \{1\to12\}
         &\text{if $a\equiv\pm1\pmod6$ and $b\equiv-1\pmod6$,} \\
    \emptyset&\text{otherwise.}\\
  \end{cases}
\]
\end{theorem}

Smyth's theorem says in particular that with at most one exception,
every arrow for a Lucas sequence has prime weight.  This is not true
for~EDS and is due to the fact that the number of points~$\#E(\FF_q)$
on an elliptic curve over a finite field varies irregularly
compared to the number of points in the multiplicative group~$\FF_q^*$
of a finite field.  This leads to EDS arrows of the form $n\to nd$,
where~$d$ is a so-called aliquot number for the EDS, as in the
following definition.  The aliquot phenomenon has no analogue in the
case of Lucas sequences.

\begin{definition}
A list~$(p_1,\ldots,p_\ell)$ of distinct primes of good
reduction for~$E$ is an \emph{aliquot cycle} for~$\EDS$ if
\[
  p_{i+1} =\min\{ r\ge1 : p_i \mid D_r\}
  \qquad\text{for all $1\le i\le \ell$,}
\]
where we set $p_{\ell+1}=p_1$ to complete the cycle.  The associated
\emph{aliquot number} is the product~$p_1\cdots p_\ell$.  
\end{definition}

The index divisibility graph of an EDS is considerably more
complicated than that of a Lucas sequence.  We state here a simplified
version of Theorem~\ref{theorem:buildarrows}, which is the main result
of this paper.  We remark that an analogue of our main result for EDS
associated to singular elliptic curves would give a version of Smyth's
theorem; see Remark~\ref{remark:singularEDS} for details.

\begin{theorem}
\label{theorem:intromainresult}
Let~$\EDS$ be a minimal regular EDS associated to the elliptic curve
$E/\QQ$ and point $P\in E(\QQ)$.  \textup(See
Section~$\ref{section:aliqseqs}$ for the definition of regularity. In
particular, every EDS has a regular subsequence.\textup)
\begin{parts}
\Part{(a)}
If $n\in\Scal(\EDS)$ and $p$ is prime and~$p\mid D_n$,
then $(n\to np)\in\Arrow(\EDS)$.
\Part{(b)}
If $n\in\Scal(\EDS)$ and $d$ is an aliquot number for~$\EDS$
and $\gcd(n,d)=1$, then $(n\to nd)\in\Arrow(\EDS)$.
\Part{(c)}
If $p\ge7$ is a prime of good reduction for~$E$ and if
$(n\to np)\in\Arrow(\EDS)$, then either $p\mid D_n$
or~$p$ is an aliquot number for~$\EDS$.
\Part{(d)}
If $gcd(n,d)=1$ and if $(n\to nd)\in\Arrow(\EDS)$ and
if~$d=p_1p_2\cdots p_\ell$ is a product of $\ell\ge2$ distinct primes
of good reduction for~$E$ satisfying $\min p_i>(2^{-1/2\ell}-1)^{-2}$,
then~$d$ is an aliquot number for~$\EDS$.
\end{parts}
\end{theorem}

We briefly describe the contents of this note. In
Section~\ref{section:edsprelims} we give some basic properties of
elliptic divisibility sequences. In particular,
Lemma~\ref{lemma:formalgp} states fairly delicate divisibility
estimates whose origins lie in the formal group of~$E$.  The brief
Section~\ref{section:aliqseqs} gives the definition of aliquot cycles
and aliquot numbers for~EDS.  Section~\ref{section:arrowsingraph}
contains the statement and proof of Theorem~\ref{theorem:buildarrows},
which is the main result of this paper.
Theorem~\ref{theorem:buildarrows}, which is an expanded version of
Theorem~\ref{theorem:intromainresult}, explains how to construct the
arrows that are used to build~$\Scal(\EDS)$.  This is followed in
Section~\ref{section:arrowremarks} with a number of remarks and
examples related to our main theorem.
Section~\ref{section:aliquotexamples} defines aliquot cycles on an
elliptic curve (see~\cite{SilvStangEAS}) and explains how they are
related to aliquot cycles for an EDS on that curve.  Finally, in
Section~\ref{section:miscremarks}, we make some miscellaneous remarks
on general index divisibility sets and on an alternative definition
of~EDS.

\section{Preliminaries on elliptic divisibility sequences}
\label{section:edsprelims}

Let~$\EDS$ be a minimal EDS associated to an elliptic curve~$E/\QQ$
and point~$P\in E(\QQ)$.  We let~$\Disc(E)$ denote the minimal
discriminant of~$E$. For all primes~$p$ we have
\[
  p\mid D_n \quad\Longleftrightarrow\quad [n]P\equiv O \pmod{p}.
\]

\begin{definition}
We write~$r_n=r_n(\EDS)$ for  the \emph{rank of apparition} of~$n$
in~$\EDS$, which is defined by
\[
  r_n = \min\{ r \ge 1 : n \mid D_r \}.
\]
Let $\Neron/\Spec\ZZ$ denote the N\'eron model of~$E$.  Then an
equivalent definition of~$r_n$ is that it is the smallest value
of~\text{$r\ge1$} such that
\[
  [r]P \equiv O \pmod{n},
\]
where the congruence takes place in $\Neron(\ZZ/n\ZZ)$.
\end{definition}

The following three lemmas contain virtually all of the information
about EDS that we will use in our analysis of EDS index divisibility.

\begin{lemma}
\label{lemma:nDmrnmmPOn}
Let~$\EDS$ be a minimal EDS associated to an elliptic curve~$E/\QQ$
and point~$P\in E(\QQ)$.  Then
\[
  n\mid D_m \quad\Longleftrightarrow\quad r_n\mid m
   \quad\Longleftrightarrow\quad [m]P\equiv O\pmod{n}.
\]
\end{lemma}
\begin{proof}
Immediate from the definitions.
\end{proof}

The next lemma describes the growth of~$p$-divisibility for
EDS. A direct corollary is that an EDS is a divisibility sequence.

\begin{lemma}
\label{lemma:formalgp}
Let~$\EDS=(D_n)_{n\ge1}$ be a minimal~EDS, let~$n\ge1$, and let~$p$ be a
prime satisfying~$p\mid D_n$. 
\begin{parts}
\Part{(a)}
For all $m\ge1$ we have
\[
  \ord_p(D_{mn}) \ge \ord_p(mD_n). 
\]
\Part{(b)}
The inequality in~\textup{(a)} is strict,
\[
  \ord_p(D_{mn}) > \ord_p(mD_n),
\]
if and only if
\[
  \def\quad{\hspace{.3em}}
  p=2
  \quad\text{and}\quad 2\mid m
  \quad\text{and}\quad \ord_2(D_n)=1
  \quad\text{and}\quad 
  \left(\begin{tabular}{@{}c@{}}
  \text{$E$ has ordinary or multi-} \\
  \text{plicative reduction at $2$} \\
  \end{tabular}\right).
\]
\textup(For the definition of ordinary reduction,
see~\cite[\S{V.3}]{AEC}.  In particular,~$E$ has ordinary reduction
at~$2$ if and only if $2\mid\#E(\FF_2)$.\textup)
\end{parts}
\end{lemma}
\begin{proof}
The assumption that~$p\mid D_n$ is equivalent to the assertion
that $[n]P$ is in~$E_1(\QQ_p)$, the kernel of reduction modulo~$p$.
We use the standard isomorphism between~$E_1(\QQ_p)$ and the formal
group~$\hat E(p\ZZ_p)$ associated to~$E$ given by
\[
  \f : E_1(\QQ_p)\longrightarrow\hat E(p\ZZ_p),\qquad
  (x,y) \longrightarrow -x/y.
\]
Note that this isomorphism is valid even if~$E$ has bad reduction
at~$p$, in which case~$\hat E$ is the formal additive or
multiplicative group. (See \cite[Chapter~IV]{AEC} for basic properties
of formal groups.)  
\par
Writing~$x([n]P)=(A_n/D_n^2,B_n/D_n^3)$, our assumption that $p\mid D_n$
implies that $p\notdivide A_nB_n$, so
\begin{equation}
  \label{eqn:ordpphinPAnBnBn}
  \ord_p\f\bigl([n]P\bigr) = \ord_p(-A_nD_n/B_n) = \ord_p(D_n).
\end{equation}
\par
Standard properties of formal groups~\cite[IV.2.3(a),~IV.4.4]{AEC} say
that the multi\-pli\-ca\-tion-by-$p$ map has the form
\begin{equation}
  \label{eqn:pEzpfzgzp}
  [p]_{\hat E}(z) = pf(z) + g(z^p),
\end{equation}
where $f,g\in\ZZ_p[[z]]$ are power series with no constant term, and $f$ has the form $f(z) = z+O(z^2)$.
It follows that for $\ord_p(z) \ge 1$, we have
\begin{equation}
  \label{eqn:ordpphEz}
  \ord_p\bigl([p]_{\hat E}(z)\bigr) \ge \ord_p(pz).
\end{equation}
We write $m=p^ks$ with $p\nmid s$.
Repeated application of~\eqref{eqn:ordpphEz} gives
\begin{equation}
  \label{eqn:ordp1}
  \ord_p\bigl([p^k]_{\hat E}(z)\bigr) \ge \ord_p(p^kz).
\end{equation}
Further, we have $[s]_{\hat E}(z)=sz+O(z^2)$, so
\begin{equation}
  \label{eqn:ordp2}
  \ord_p\bigl([sp^k]_{\hat E}(z)\bigr) = \ord_p([p^k]_{\hat E}z).
\end{equation}
Combining~\eqref{eqn:ordp1} and~\eqref{eqn:ordp2} gives
\[
  \ord_p\bigl([m]_{\hat E}(z)\bigr) \ge \ord_p(p^kz),
\]
with equality if~$k=0$.  Substituting~$z=\f\bigl([n]P\bigr)$
and using~\eqref{eqn:ordpphinPAnBnBn} gives~(a), and it also gives~(b) if
$p\nmid m$.
\par
To prove~(b) in general, we assume that $p\mid m$.
Analyzing~\eqref{eqn:pEzpfzgzp} more closely, we see that
\begin{equation}
  \label{eqn:ordpphEz2}
  \ord_p\bigl([p]_{\hat E}(z)\bigr) = \ord_p(pz)=\ord_p(z)+1
\end{equation}
unless $\ord_p(pz) = p\ord_p(z)$. (Note that $\ord_p(z)\ge1$.) Since
\[
  \ord_p(pz)=p\ord_p(z)
  \quad\Longleftrightarrow\quad
  1=(p-1)\ord_p(z),
\]
we see that~\eqref{eqn:ordpphEz2} holds except possibly in the
case $p=2$ and $\ord_p(z)=1$.  
\par
Suppose now that $p=2$ and $\ord_p(z)=1$. 
The formal group law for an elliptic curve starts~\cite[\S{IV.1}]{AEC}
\[
  [2]_{\hat E}(z) = 2z - a_1z^2 -2a_2z^3 + (a_3+a_1a_2)z^4 + \cdots,
\]
where~$a_1,\ldots,a_6$ are Weierstrass coefficients. 
Hence under the assumption that $\ord_2(z)\ge1$, we see
that~\eqref{eqn:ordpphEz2} fails if and only if
\begin{align*}
  \ord_2\bigl(2z-a_1z^2+O(2z^3)\bigr) \ge 3
  &\iff
  \ord_2(1-a_1z/2) \ge 1\\
  &\iff
  \ord_2(a_1)=0,\quad\text{i.e., $a_1\in\ZZ_2^*$.}
\end{align*}
(The last implication follows because $z\equiv2\pmod4$, so
$1-a_1z/2\equiv1-a_1\pmod2$.)  If~$E$ has good reduction modulo~$2$,
then $j(E)\equiv a_1^{12}/\Disc(E)\pmod2$, so~\cite[Exer.~5.7]{AEC}
gives
\[
  \ord_2(a_1)=0 \iff j(E)\not\equiv0\pmod2 \iff 
  \text{$E$ is ordinary mod 2.}
\]
On the other hand, if $E$ has bad reduction at~$2$, then an easy
computation shows that $a_1\equiv1\pmod2$ for multiplicative reduction
and $a_1\equiv0\pmod2$ for additive reduction. This completes the
proof that~\eqref{eqn:ordpphEz2} fails if and only if $p=2$ and $p\mid
m$ and $\ord_p(D_n)=1$ and $E$ has either ordinary or multiplicative
reduction. We call this the exceptional case.
\par
Repeated application of~\eqref{eqn:ordpphEz2} shows that if we are not
in the exceptional case, then
\[
  \ord_p\bigl([p^k]_{\hat E}(z)\bigr)  =\ord_p(z)+k.
\]
In the exceptional case, the first multiplication by~$[p]$ gives
a strict inequality, after which we are out of the exceptional case
and can apply~\eqref{eqn:ordpphEz2}, so we find that
\[
  \ord_p\bigl([p^k]_{\hat E}(z)\bigr)  
  = \ord_p\bigl([p]_{\hat E}(z)\bigr)+k-1 > \ord_p(z)+k.
\]
Now using~\eqref{eqn:ordp2} and the fact that $m=p^ks$ with $p\nmid
s$, we get
\[
  \ord_p\bigl([m]_{\hat E}(z)\bigr)  > \ord_p(mz)
\]
in the exceptional case and
\[
  \ord_p\bigl([m]_{\hat E}(z)\bigr)  =\ord_p(mz)
\]
otherwise.  Substituting~$z=\f\bigl([n]P\bigr)$ and
using~\eqref{eqn:ordpphinPAnBnBn} proves~(b).
\end{proof}

The third lemma gives bounds for~$r_p$.

\begin{lemma}
\label{lemma:weilbd}
Let~$\EDS$ be a minimal EDS associated to an elliptic curve~$E/\QQ$
and point~$P\in E(\QQ)$ and let~$p$ be a prime.  Then
\[
  r_n\mid\#\Neron(\ZZ/n\ZZ).
\]
In particular, if~$p$ is a prime with $P\in E_\ns(\FF_p)$, then
\[
  r_p \le (\sqrt{p}+1)^2.
\]
If $P\in E_\ns(\FF_p)$ and $E$ has bad reduction at~$p$, then~$r_p$
divides~$p-1$, $p+1$, or $p$ depending respectively on whether the
reduction is split multiplicative, non-split multiplicative, or
additive.
\end{lemma}
\begin{proof}
The first statement is immediate, since~$r_n$ is the order of the
point~$P$ in the group~$\Neron(\ZZ/n\ZZ)$.  The estimates for~$r_p$
follow from the Hasse--Weil bound $\#E(\FF_p)\le(\sqrt p+1)^2$
when~$E$ has good reduction, and the explicit description
of~$E_\ns(\FF_p)$ for the three types of bad reduction. 
\end{proof}

\begin{example}
The minimal~EDS associated to
\[
  E: y^2 + xy = x^3 - 2x + 1
  \quad\text{and}\quad
  P = (1,0)
\]
is the sequence
\[
  1,1,1,2,1,3,7,8,25,37,\ldots\,.
\]
Thus~$D_4=2$ and $D_8=8$, so
\[
  3 = \ord_2(D_{2^3}) > \ord_2(2D_{2^2}) = 2.
\]
The strictness of the inequality in Lemma~\ref{lemma:formalgp}(a)
corresponds to the exceptional case $p=2$, $m=2$, and $n=4$, where we
note that $\ord_2(D_4)=1$ and~$\#E(\FF_2)=4$, so in particular~$E$ has
ordinary reduction at~$2$.
\end{example}

\begin{remark}
More generally, for any integer~$N\ge2$ there exists a minimal~EDS
such that
\[
  \ord_2(D_{2}) = \ord_2(D_1) + N.
\]
Here is one construction.  Choose an elliptic curve of positive rank
having a rational $2$-torsion point $T$ in the formal group~$\hat
E(2\ZZ_2)$. Taking a multiple of a point of infinite order, we can
find a rational nontorsion point~$Q$ in~$\hat E(2^{N}\ZZ_2)$. Then
the EDS associated to $P=Q+T$ will have $\ord_2(D_1)=1$ and
$\ord_2(D_2)=N+1$. The reason that this only works for the prime $p=2$
is because for $p\ge3$, the formal group~$\hat E(\ZZ_p)$ is torsion
free; in fact, it is isomorphic to the additive group~$\ZZ_p^+$.
\end{remark}

\begin{proposition}
Let $\EDS$ be a minimal EDS. 
\begin{parts}
\Part{(a)}
$\EDS$ is a divisibility sequence.
\Part{(b)}
The set $\Scal(\EDS)$ is closed under multiplication.
\end{parts}
\end{proposition}
\begin{proof}
(a)\enspace
We need to prove that $D_m\mid D_{mn}$. It suffices to prove
that $\ord_p(D_{mn})\ge\ord_p(D_n)$ for all primes~$p$, but this is 
immediate from Lemma~\ref{lemma:formalgp}(a).
\par\noindent(b)\enspace
Suppose that $m,n\in\Scal(\EDS)$ and let~$p\mid n$. Then $p\mid D_n$,
so Lemma~\ref{lemma:formalgp}(a) and the assumption that $n\mid D_n$
give
\[
  \ord_p(D_{mn}) \ge \ord_p(mD_n) \ge \ord_p(mn).
\]
Reversing the roles of~$m$ and~$n$ for $p\mid m$ again gives
$\ord_p(D_{mn})\ge\ord_p(mn)$. Hence~$mn\mid D_{mn}$,
so~$mn\in\Scal(\EDS)$.
\end{proof}
 
\begin{remark}
If $p\ge3$ and $p\mid D_1$, then Lemma~\ref{lemma:formalgp}(b) with~$n=1$
says that $\ord_p(D_m)=\ord_p(mD_1)$ for all~$m$.
\end{remark}

\begin{remark}
Although we will not need this fact, we mention that elliptic
divisibility sequences grow extremely rapidly.  Thus if~$\EDS$ is
associated to~$(E,P)$, then
\[
  \lim_{n\to\infty} \frac{\log|D_n|}{n^2} = \hhat_E(P),
\]
where~$\hhat_E(P)>0$ is the canonical height of~$P$~\cite[VIII~\S9]{AEC}.
\end{remark}

\section{Aliquot Cycles and Aliquot Numbers for EDS}
\label{section:aliqseqs}

In this section we define aliquot cycles and aliquot numbers
associated to an EDS.

\begin{definition}
Let $\EDS$ be an EDS associated to the curve~$E(\QQ)$ and point $P\in
E(\QQ)$.  We recall that~$r_n(\EDS)$ denotes the rank of apparition
of~$n$ in the sequence~$\EDS$; see Section~\ref{section:edsprelims}.
An \emph{aliquot cycle} (\emph{of length~$\ell$})
for~$\EDS$ is a sequence $(p_1,p_2,\ldots,p_\ell)$ of distinct primes
of good reduction for~$E$ such that
\[
  r_{p_1}(\EDS) = p_2,\semiquad
  r_{p_2}(\EDS) = p_3,\semiquad\ldots,\semiquad
  r_{p_{\ell-1}}(\EDS) = p_\ell,\semiquad
  r_{p_\ell}(\EDS) = p_1.
\]
An \emph{amicable pair} is an aliquot cycle of length two.
\par
If we drop the requirement that~$E$ have good reduction, then
we call  $(p_1,p_2,\ldots,p_\ell)$ a \emph{generalized aliquot cycle}.
\end{definition}

In our study of index divisibility for EDS, the products of the primes
appearing in each aliquot cycle play a key role, so we give them a
name.

\begin{definition}
Let~$\EDS$ be a minimal EDS.  We define the set of \emph{aliquot
  numbers} of~$\EDS$ to be
\[
  \Aliq(\EDS) = \{p_1\cdots p_\ell :
   \text{$(p_1,\ldots,p_\ell)$ is an aliquot cycle for $\EDS$}\}.
\]
We also define the larger set
\[
  \AliqGen(\EDS)  = \{p_1\cdots p_\ell :
   \text{$(p_1,\ldots,p_\ell)$ is a generalized aliquot cycle for $\EDS$}\}.
\]
\end{definition}

\begin{remark}
We observe that an aliquot cycle of length one consists of a single
prime~$p$ satisfying $r_p(\EDS)=p$. If $p\ge7$, Hasse's estimate for
$\#E(\FF_p)$ tells us that
\[
  r_p(\EDS)=p \iff \#E(\FF_p)=p.
\]  
Thus in standard terminology, the primes $p\ge7$ in~$\Aliq(\EDS)$ are
exactly the \emph{anomalous primes} for the elliptic curve~$E$.  
\end{remark}

\section{Arrows in the Index Divisibility Graph}
\label{section:arrowsingraph}

This section contains our main results. In
Theorem~\ref{theorem:buildarrows} we classify the
arrows~$(n\to nd)\in\Arrow(\EDS)$ for a large class of~EDS,
as described in the following definition.

\begin{definition}
Let~$\EDS$ be a minimal EDS associated to the elliptic curve and point
$(E,P)$. We say that $D$ is \emph{$2$-irregular} if the following five
irregularity conditions are true:
\begin{center}
  \begin{tabular}{lll}
    (\IR1) $E$ has good reduction at $2$,
    &(\IR2) $\#E(\FF_2)=4$,
    &(\IR3) $r_2=4$,\\[1\jot]
    (\IR4) $D_2$ is odd,
    &(\IR5) $\ord_2(D_4)=1$,
  \end{tabular}
\end{center}
If any of the conditions (\IR1)--(\IR5) is false, then we say that $\EDS$ is
\emph{$2$-regular}. If in addition we have
\[
  P\in E_\ns(\FF_p)\qquad\text{for all primes $p\mid\Disc(E)$,}
\]
then we simply say that~$\EDS$ is \emph{regular}.
\end{definition}

\begin{remark}
Our main result, Theorem~$\ref{theorem:buildarrows}$, gives a good
description of the index divisibility graph~$\Scal(\EDS)$ for
regular~EDS. Our decision to restrict attention to regular EDS
represents a compromise between our desires for generality and
conciseness, as well as the need to keep our exposition to a
reasonable length. We remark that much of our analysis goes through
for non-regular~EDS, in the sense that
Theorem~$\ref{theorem:buildarrows}$ is still true for many (but
generally not all) values of~$n$, and that a long case-by-case
analysis would give a lengthy statement that applies to most (maybe
even all) values of~$n$. In any case, we note that
every~$\EDS=(D_n)_{n\ge1}$ contains a regular subsequence
$\EDS'=(D_{nk})_{n\ge1}$, and then Theorem~$\ref{theorem:buildarrows}$
applies to this subsequence.
\end{remark}

We start with a description of the index divisibility set of an~EDS
that will be a key tool for our classification. Its proof uses only
the formal group properties of an EDS (Lemma~\ref{lemma:formalgp}).

\begin{proposition}
\label{proposition:amiq}
Let~$\EDS$ be a minimal regular EDS associated to the elliptic curve
and point $(E,P)$.  Then the following are equivalent\textup:
\begin{parts}
\Part{(a)}
$n\mid D_n$, i.e, $n\in\Scal(\EDS)$.
\Part{(b)}
There is some exponent~$e\ge1$ such that $n\mid D_n^e$.
\Part{(c)}
Every prime dividing $n$ also divides $D_n$.
\Part{(d)}
For all primes~$p$, we have $p\mid n \implies r_p\mid n$.
\end{parts}
\end{proposition}
\begin{proof}
Statements~(b) and~(c) are obviously equivalent, and~(c)
and~(d) are equivalent by Lemma~\ref{lemma:nDmrnmmPOn}.
It is also clear that~(a) implies~(b).
It remains to show that~(b) implies~(a), i.e., that
\[
  n \mid D_n^e \longrightarrow n\mid D_n.
\]
It suffices to prove that for all primes~$p$ we have
\begin{equation}
  \label{eqn:pgcdnDn}
  p \mid \gcd(n,D_n) \implies \ord_p(D_n) \ge \ord_p(n).
\end{equation}
So we let $p$ be a prime dividing both~$n$ and~$D_n$ and we write $n =
p^\nu k$ with $p\nmid k$ and $\nu \ge 1$.  If~$\nu=1$,
then~\eqref{eqn:pgcdnDn} is obviously true (note $p\mid D_n$), so we
may assume that $\nu\ge2$.
\par
We consider first the case that $p\mid D_{pk}$. Applying
Lemma~\ref{lemma:formalgp}(a) to $D_n=D_{p^{\nu-1}\cdot pk}$, we obtain
\[
  \ord_p(D_n)
  = \ord_p(D_{p^{\nu}k})
  \ge \ord_p(p^{\nu-1}D_{pk})
  \ge \nu
  = \ord_p(n).
\]
This shows that~\eqref{eqn:pgcdnDn} is true in this case.
\par
We next suppose that $p\nmid D_{pk}$, and we will show that
either~\eqref{eqn:pgcdnDn} is true or else~$\EDS$ is $2$-irregular.
The assumption that $p\nmid D_{pk}$ is equivalent to $r_p\nmid
pk$. But we are assuming that~$p\mid D_{p^{\nu}k}$, so we have
$r_p\mid p^{\nu}k$. It follows that $p^2\mid r_p$, which is a very
strong condition. In particular, since the regularity assumption
implies that that $P\in E_\ns(\FF_p)$, and since~$r_p$ is the order
of~$P$ in~$E(\FF_p)$, we find that
\[
  p^2 \mid \#E_\ns(\FF_p).
\]
Hence~$E$ has nonsingular reduction modulo~$p$, and using the
Hasse--Weil estimate, we further deduce that~$p=2$
and~$r_2=\#E(\FF_2)=4$. This gives 
conditions~(\IR1),~(\IR2), and~(\IR3) in the definition of
$2$-irregu\-la\-rity. Further,~$D_2$ must be odd, since
otherwise~$r_2$ would divide~$2$, so we get condition~(\IR4).
\par
Since $2\mid D_4$, so $2\mid D_{4k}$, we can apply
Lemma~\ref{lemma:formalgp}(a) to $D_n=D_{2^{\nu-2}\cdot 4k}$ to obtain
\[
  \ord_2(D_n)
  = \ord_2(D_{2^{\nu}k})
  \ge \ord_2(2^{\nu-2}D_{4k})
  = \ord_2(D_{4k}) + \nu - 2.
\]
If $\ord_2(D_4)\ge2$, then this implies that
$\ord_2(D_n)\ge\ord_2(n)$, so~\eqref{eqn:pgcdnDn} is true and we are
done. Otherwise $\ord_2(D_4)=1$ and we have verified condition~(\IR5)
for~$\EDS$ to be $2$-irregular. This is a contradiction, since
we have assumed that~$\EDS$ is regular, which completes the
proof of  Proposition~\ref{proposition:amiq}.
\end{proof}

We are now ready to state and prove our main theorem.

\begin{theorem}
\label{theorem:buildarrows}
Let~$\EDS$ be a minimal regular EDS associated to the elliptic curve and
point $(E,P)$.  
\begin{parts}
\Part{(a)}
Let $n\ge1$. Then
\[
  n\in\Scal(\EDS)\text{ and }
  \left(\begin{tabular}{@{}c@{}}
  \text{$p\mid D_n$ or $E$ has} \\
  \text{additive reduction} \\
  \end{tabular}\right)
  \implies
    (n\to np)\in\Arrow(\EDS).
\]
\Part{(b)} 
Let $n\ge1$ and $d\ge1$. Then
\[
  n\in\Scal(\EDS)\text{\quad and\quad} d\in\AliqGen(\EDS)
  \implies
  \left(\begin{tabular}{@{}l@{}}
    $nd\in\Scal(\EDS)$ and\\ $\gcd(d,n)=1$ or $d$\\
  \end{tabular}\right).
\]
Furthermore, 
\begin{align*}
  n\in\Scal(\EDS)\text{\quad and\quad} d\in\AliqGen(\EDS)
  &\text{\quad and\quad} \gcd(d,n)=1\\
  &\implies
  (n\to~nd)\in~\Arrow(\EDS).
\end{align*}
\Part{(c)}
Let $n\ge1$ and let $p$ be a prime such that
\[
  n\in\Scal(\EDS),\quad
  p\nmid D_n,\quad\text{and}\quad
 (n\to np)\in\Arrow(\EDS).
\]
\vspace{-10pt}
\begin{parts}
\Part{(1)}
If $E$ has good reduction at $p$ and $\#E(\FF_p)\ne2p$, then
\[
  p\in\Aliq(\EDS).
\]
\textup(If $p\ge7$, then we always have $\#E(\FF_p)\ne2p$.\textup)
\Part{(2)}
If $E$ has bad reduction at~$p$, then
\[
  \text{$E$ has additive reduction at $p$.}
\]
\end{parts}
\Part{(d)}
Let $n\ge1$ and $d\ge1$ with $d$ composite.
Define
\[
  t = (\textup{number of primes~$p\mid d$ such that~$r_p$ is composite}),
\]
and 
\[
  p_0 = (\textup{smallest prime divisor of~$nd$}).
\]
Suppose that
\[
  (n\to nd)\in\Arrow(\EDS).
\]
Then one of the following statements is true\textup:
\begin{parts}
\Part{(i)}
$t=0$  and $d\in\AliqGen(\EDS)$.
\Part{(ii)} 
$t\ge1$ and
\begin{equation}
  \label{eqn:minpigt212l12x} 
  \prod_{p\mid d} \left(1+\frac{1}{\sqrt{p}}\right)^2
  \ge \prod_{p\mid d} \frac{\#E(\FF_{p})}{p} 
  \ge p_0^t.
\end{equation}
\end{parts}
\end{parts}
\end{theorem}

\begin{proof}
(a)\enspace
Suppose first that~$n\in\Scal(\EDS)$ and~$p\mid D_n$.
Write $n=p^ik$ with $p\notdivide k$.  Then
\begin{align*} 
  \ord_p(D_{np}) &\ge \ord_p(D_n) + 1
    &&\text{from Lemma~\ref{lemma:formalgp}(a),} \\
  &\ge \ord_p(n) + 1
    &&\text{since $n\in\Scal(\EDS)$, i.e., $n\mid D_n$,}\\
  &= i+1.
\end{align*}
Further, $k\mid n\mid D_n\mid D_{np}$. Hence $p^{i+1}k\mid D_{np}$,
i.e., $np\mid D_{np}$, so $np\in\Scal(\EDS)$. And since there are
no proper divisors between~$n$ and~$np$, it follows that the 
directed graph~$\Scal(\EDS)$ contains the arrow $n\to np$.
\par
Next we consider the case that
$n\in\Scal(\EDS)$ and $p\nmid D_n$ and $E$ has additive
reduction.  Additive reduction implies that $\#E_{\ns}(\FF_p)=p$, so that $r_p \mid p$ and $p \mid D_p \mid D_{np}$.  Meanwhile, $n \mid D_n \mid D_{np}$ by assumption.  Since $p \nmid D_n$, it must be that $p \nmid n$.  Hence $np \mid D_{np}$ and $np \in \Scal(\EDS)$, from which we conclude that $(n \to np) \in \Arrow(\EDS)$.
\par\noindent(b)\enspace
Let $d=p_1\cdots p_\ell\in\AliqGen(\EDS)$ be a
generalized aliquot number
for~$\EDS$. We will show that $nd\in\Scal(\EDS)$. First, let $p=p_i$ be one of
the primes dividing~$d$.  
The rank of apparition satisfies $r_{p_i}=p_{i+1}$,
where for notational convenience we let
$p_{\ell+1}=p_1$. Hence $p_i \mid D_{p_{i+1}} \mid D_{nd}$.  
Next let $p$ be a prime dividing $n$.  Then $p \mid n \mid D_n
\mid D_{nd}$.  We have shown that any prime $p$ dividing $nd$
satisfies $p \mid D_{nd}$.  By Proposition \ref{proposition:amiq}, we
conclude that $nd \in \Scal(\EDS)$.
\par
Now we determine the possible values of $\gcd(d,n)$.  If $p_i \mid n$,
then since $n \in \Scal(\EDS)$, Proposition \ref{proposition:amiq}
implies that $p_{i+1} = r_{p_i}$ must divide $n$.  Hence, by the
construction of $d$, we see that either
$n$ is divisible by none of the
prime dividing $d$, or it is divisible by all of them.  Therefore
$\gcd(d,n)= 1$ or $d$.
\par
Suppose now that $\gcd(d,n)=1$ and that $e\mid d$ is a divisor
of~$d$ such that $ne \in \Scal(\EDS)$.
Then by the reasoning of the last paragraph, with $n$  replaced
by $ne$, we find that $\gcd(d,ne)=1$ or $d$. But $\gcd(d,n)=1$
and~$e\mid d$, so we conclude that~$e=1$ or~$e=d$. 
Hence $(n\to nd)\in\Arrow(\EDS)$ by definition.
\par
This completes the proof of~(b).  We also note that some condition
such as $\gcd(n,d)=1$ is necessary. For example, suppose that $(p,q)$
is an amicable pair and that $p$ divides $D_n$. Then there is no arrow
from $n$ to $npq$, because there are ``shorter'' arrows $n\to np\to
npq$.
\par\noindent(c)\enspace
We are given that $n\in\Scal(\EDS)$, $np\in\Scal(\EDS)$, and
$p\notdivide D_n$.  Since $np\in\Scal(\EDS)$, Proposition
\ref{proposition:amiq} implies $p \mid D_{np}$. We observe that
\begin{align*}
  p\nmid D_n
  &\quad\Longleftrightarrow\quad 
  [n]P\not\equiv O\pmod{p}, \\
  p\mid D_{np}
  &\quad\Longleftrightarrow\quad
  [np]P\equiv O\pmod{p}.
\end{align*}
Hence under our assumptions, in particular the regularity assumption,
we see that the point~$[n]P$ has order exactly~$p$ 
in~$P\in E_\ns(\FF_p)$. Hence $p\mid E_\ns(\FF_p)$.
\par\noindent(c-1)\enspace
Suppose first that~$E$ has good reduction at~$p$, so
$E_\ns(\FF_p)=E(\FF_p)$.  We want to show that $p\in\Aliq(\EDS)$.
The assumption that~$\#E(\FF_p)\ne2p$, combined with Hasse's estimate
\text{$|p+1-\#E(\FF_p)|\le2\sqrt{p}$}, implies that
\begin{equation}
  \label{eqn:Hasse}
  p \mid \#E(\FF_p)
  \quad\Longleftrightarrow\quad 
  \#E(\FF_p)=p.
\end{equation}
Since~$r_p\mid\#E(\FF_p)$, we see  that $r_p=1$ or $r_p=p$. 
But $r_p=1$ implies that $p\mid D_1$, contradicting $p\notdivide D_n$.
Therefore $r_p=p$, which implies that $p\in\Aliq(\EDS)$, i.e.,~$(p)$ is
an aliquot cycle of length one.
\par\noindent(c-2)\enspace
Next suppose that~$E$ has bad reduction at~$p$.
It follows from~$p\mid E_\ns(\FF_p)$ that~$E$ has additive reduction
at~$p$. (If it had multiplicative reduction, then~$E_\ns(\FF_p)$ would
contain~$p\pm1$ points, depending on whether the reduction is split
or nonsplit.)
\par\noindent(d)\enspace
We first show that $\gcd(D_n,d)=1$, which in particular implies
that~$\gcd(n,d)=1$, since~$n\mid D_n$. To see this, suppose
to the contrary
that \text{$\gcd(D_n,d)>1$}, and let~$p$ be a prime dividing~$\gcd(D_n,d)$.
Since \text{$p\mid D_n$}, we know from~(a) that~$(n\to np)\in\Arrow(\EDS)$.
But since \text{$p\mid d$}, we have divisibilities \text{$n\mid np\mid nd$},
so the fact that $n\to np$ and $n\to nd$ are arrows implies that
either $n=np$ or $np=nd$. Neither of these is possible, since $p\ge2$,
and~$d$ is composite by assumption. This completes the proof that
$\gcd(D_n,d)=1$.
\par
In order to analyze the arrow~$(n\to nd)$, we associate to the
integer~$d$ a directed graph~$\Gcal_d$ as in the following lemma.
The graph~$\Gcal_d$ classifies the primes dividing each rank
of apparition~$r_p$.

\begin{lemma}
\label{lemma:inoutarrows}
Let~$\EDS$ be a minimal regular EDS,
let $n\ge1$ and $d\ge1$ with $d$ composite, and assume that
\text{$(n\to nd)\in\Arrow(\EDS)$}.  We construct a directed
graph~$\Gcal_d$ with vertices and arrows defined as
follows\textup:
\begin{align*}
  \Vertex(\Gcal_d) &= \{\text{primes $p$ such that $p\mid d$}\},\\
  \Arrow(\Gcal_d) &= \{ p \to q : q\mid d~\text{and}~q \mid r_p \}.
\end{align*}
\textup(N.B., the graph~$\Gcal_d$ is entirely distinct from the graph
on~$\Scal(\EDS)$.\textup) 
\begin{parts}
\Part{(a)}
Every vertex of~$\Gcal_d$ has  an in-arrow.
\Part{(b)}
Every vertex of~$\Gcal_d$ has  an out-arrow.
\Part{(c)}
The graph~$\Gcal_d$ is connected.
\end{parts}
\end{lemma}
\begin{proof}
Let $q\mid d$ be a prime divisor of~$d$, i.e.,~$q$ is a vertex of~$\Gcal_d$.
\par\noindent(a)\enspace
We need to show that~$q$ has an in-arrow.  Let $d' = d/q$.  Since $(n
\rightarrow nd) \in \Arrow(\EDS)$, we know that $nd' \notin
\Scal(\EDS)$.  By Proposition~\ref{proposition:amiq}, this implies the existence
of a prime $p \mid nd'$ satisfying $r_p \nmid nd'$.  Since $r_p \mid
nd$ by Proposition~\ref{proposition:amiq}, this implies that $q \mid r_p$, which
shows $p \rightarrow q$ as required.
\par\noindent(b)\enspace
We need to show that~$q$ has an out-arrow.  Since $q \mid d$, we have
$q \mid D_d$ or $r_q \mid d$.  This shows that some prime $p \mid d$
satisfies $p \mid r_q$, and thus $q \rightarrow p$ as required.
\par\noindent(c)\enspace
Define $d'$ to be the part of $d$ supported on primes
appearing as vertices in a connected component $\Gcal'$ of $\Gcal_d$.
Then for each prime $p \mid d'$, all primes dividing $r_p$ appear in
$\Gcal'$ by connectedness.  Since $r_p \mid nd$ by the assumption that
$nd \in \Scal(\EDS)$, this implies $r_p \mid nd'$.  This shows that
$nd' \in \Scal(\EDS)$, which contradicts $(n \rightarrow nd) \in
\Arrow(\EDS)$ unless $d'=1$ or $d'=d$.  So $\Gcal_d$ is connected.
\end{proof}

\par\noindent 

Suppose first that~$r_p$ is prime for every~$p\mid d$. We need to
prove that \text{$d\in\AliqGen(\EDS)$}.  By definition, for every
arrow $(p\to q)\in\Arrow(\Gcal_d)$ we have $q\mid r_p$, so the
assumption that~$r_p$ is prime implies that $r_p=q$.  In particular,
every vertex in the finite directed graph~$\Gcal_d$ has at most one
outgoing arrow.  But Lemma~\ref{lemma:inoutarrows}(b) tells us that
every vertex in~$\Gcal_d$ has at least one outgoing arrow, 
and Lemma~\ref{lemma:inoutarrows}(c) says that the graph is connected.
It follows that~$\Gcal_d$ consists of a single loop,
\par\noindent
\begin{picture}(200,50)(28,0)
\put(150,30){\mbox{
  $p_1 \to p_2 \to p_3 \to \cdots \to p_{t}$}}
\qbezier(170,20)(218,0)(268,20)
\put(170,20){\vector(-3,1){10}}
\end{picture}
\par\noindent
This loop satisfies $r_{p_i}=p_{i+1}$, so by definition
$(p_1,\ldots,p_t)$ is a generalized aliquot cycle for~$\EDS$, and
hence~$p_1p_2\cdots p_t\in\AliqGen(\EDS)$. Since we also know
that~$\gcd(d,n)=1$, it follows from part~(c) of the theorem that
\[
  (n\to np_1p_2\cdots p_t)\in\Arrow(\EDS).
\]
But $np_1p_2\cdots p_t\mid nd$, so that fact that $(n\to nd)\in\Arrow(\EDS)$
implies that $d=p_1p_2\cdots p_t$. Hence~$d\in\AliqGen(\EDS)$. This completes
the proof of part~(i).
\par
In order to analyze the case that one or more of the~$r_p$ are composite,
for each vertex $q\in\Gcal_d$ we let
\[
  \In(q) = \#\bigl\{p\mid d : (p\to q)\in\Arrow(\Gcal_d)\bigr\}
\]
denote the in-degree of~$q$, i.e., the number of arrows pointing in
to~$q$; and similarly~$\Out(q)$ will denote the out-degree of~$q$.
Lemma~\ref{lemma:inoutarrows} tells us that~$\In(q)\ge1$ for
all~$q\in\Gcal_d$.
For each~$p\in\Gcal_d$ we know that~$r_p$ is divisible by 
the primes at the tips of
the outgoing arrows from~$p$, so we can factor~$r_p$ as
\[
  r_p = \Bigl(\prod_{(p\to q)\in\Arrow(\Gcal_d)} q\Bigr)M_p
  \qquad\text{for some $M_p\ge1$.}
\]
Further, from Proposition~\ref{proposition:amiq}, the fact
that~$nd\in\Scal(\EDS)$ and~$p\mid d$ implies that~$r_p\mid nd$,
so every prime divisor of~$M_p$ is also a prime divisor of~$nd$.
\par
We now multiply over all~$p\in\Gcal_d$, i.e., over all~$p\mid d$, and
rearrange the terms to deduce that
\[
  \prod_{p\mid d} r_p
  = \prod_{p\mid d} \biggl(
        \prod_{\substack{\text{$q\mid d$ such that}\\(p\to q)\in\Arrow(\Gcal_d)\\}} q
    \biggr)M_p
  = \Bigl(\prod_{q\mid d} q^{\In(q)}\Bigr)
     \Bigl(\prod_{p\mid d} M_p\Bigr).
\]
Since $\In(q)\ge1$ for every~$q\mid d$, we can rewrite this as
\begin{equation}
  \label{eqn:pdrppqIn}
  \prod_{p\mid d} \frac{r_p}{p}
  =  \Bigl(\prod_{q\mid d} q^{\In(q)-1}\Bigr)
     \Bigl(\prod_{p\mid d} M_p\Bigr),
\end{equation}
where the right-hand side is a positive integer.  Using the
Hasse--Weil bound $(\sqrt{p}+1)^2\ge \#E_\ns(\FF_p)$ and the fact that
$r_p\mid\#E_\ns(\FF_p)$, we obtain the useful inequalities
\begin{equation}
  \label{eqn:pdrppqIn2}
  \prod_{p\mid d}\left(1+\frac{1}{\sqrt{p}}\right)^2
  \ge \prod_{p\mid d} \frac{\#E_\ns(\FF_p)}{p}
  \ge  \Bigl(\prod_{q\mid d} q^{\In(q)-1}\Bigr)
     \Bigl(\prod_{p\mid d} M_p\Bigr).
\end{equation}
\par
We now use~\eqref{eqn:pdrppqIn2} to derive a bound that depends on the
number of composite~$r_p$. (See also Remark~\ref{remark:precbds}.)
Let~$p_0$ be the smallest prime divisor of~$nd$. Then
\begin{align}
  \label{eqn:invsout}
  \smash[b]{
    \prod_{q\mid d} q^{\In(q)-1}
    \ge \prod_{q\mid d} p_0^{\In(q)-1}
  }
  &= p_0^{\sum_{q\mid d} (\In(q)-1)} \notag \\
  &= p_0^{\#\Arrow(\Gcal_d)-\#\Vertex(\Gcal_d)} \notag \\
  &= p_0^{\sum_{p\mid d} (\Out(p)-1)}.
\end{align}
Now consider a prime~$p\in\Gcal_d$ such that~$r_p$ is composite.
If~$r_p$ is divisible by two or more primes that also divide~$d$,
then~$\Out(p)\ge2$, so we get a factor of~$p_0$
in~\eqref{eqn:invsout}.  On the other hand, if there is some~$q\mid
r_p$ with~$q\nmid d$, then~$q\mid M_p$, so we get a factor of~$q$
in~\eqref{eqn:pdrppqIn2}.  Further, we must have~$q\mid n$, since as
noted earlier, $r_p\mid nd$.  Thus $q\ge p_0$. This proves that every
composite~$r_p$ with~$p\mid d$ contributes a factor
to~\eqref{eqn:pdrppqIn2} that is greater than or equal to~$p_0$.
Hence the lower bound in~\eqref{eqn:pdrppqIn2} is at least~$p_0^t$,
where~$t$ is the number of~$p\mid d$ such that~$r_p$ is composite.
\end{proof}

The following corollary may be compared with Smyth's
result~\cite[Corollary~2]{Smyth} for Lucas sequences.

\begin{corollary}
Let~$\EDS$ be a minimal regular EDS, let $n\in\Scal(\EDS)$, and
let~$m$ be an integer of the form
\[
  m=p_1p_2\cdots p_s\cdot d_1d_2\cdots d_t,
\]
where the primes~$p_i$  and integers~$d_i$ satisfy
\[
  \text{$p_i\mid D_n$ \quad and\quad $d_i\in\AliqGen(\EDS)$.}
\]
Then $nm\in\Scal(\EDS)$.
\end{corollary}
\begin{proof}
This is immediate from Theorem~\ref{theorem:buildarrows}(a,b)
and induction on the number of factors of~$m$.
\end{proof}

\section{Remarks on Arrow Construction}
\label{section:arrowremarks}

In this section we make a number of remarks concerning the existence
of index divisibility arrows as described in
Theorem~\ref{theorem:buildarrows}, and we give examples of
non-standard arrows as per Theorem~\ref{theorem:buildarrows}(d-ii).
We assume throughout that our EDS is minimal and regular.

\begin{remark}
\label{remark:nonstandardarrows}
Given an element~$n\in\Scal(\EDS)$, Theorem~\ref{theorem:buildarrows}
gives two ``standard'' ways to create arrows $(n\to
nd)\in\Arrow(\EDS)$.  First, Theorem~\ref{theorem:buildarrows}(a)
gives an arrow $(n\to np)$ for each prime $p\mid D_n$. Second,
Theorem~\ref{theorem:buildarrows}(b) gives an arrow for each aliquot
number~$d\in\AliqGen(\EDS)$ that is prime to~$n$.  Conversely,
Theorem~\ref{theorem:buildarrows}(d) implies that any ``non-standard''
arrow satisfies
\begin{equation}
  \label{eqn:pd11psqrt2}
  \prod_{p\mid d} \left(1+\frac{1}{\sqrt{p}}\right) \ge \sqrt{2}.
\end{equation}
In particular, writing~$\nu(d)$ for the number of distinct prime
divisors of~$d$ and $p_{\min}(d)$ for the smallest prime dividing~$d$, 
we have
\[
  \nu(d) \ge \frac{\frac12\log2}{\log(1+p_{\min}(d)^{-1/2})}
  = \frac{\log2}2\sqrt{p_{\min}(d)} + O(1).
\]
Thus if the smallest prime divisor of~$d$ is large, then~$\nu(d)$ will
be large, and~$d$ will be enormous.  The following brief table
uses~\eqref{eqn:pd11psqrt2} to give the smallest values of~$\nu(d)$
and~$d$ for various values of~$p_{\min}(d)$.
\[
  \begin{array}{|c||*{5}{c|}} \hline
    p_{\min}(d) \ge{} & 10 & 10^2 & 10^3 & 10^4 & 10^5 \\ \hline
    \nu(d) \ge{} & 2 & 4 & 12 & 36 & 100 \\ \hline
    d  \ge{}& 143 & 1.21\cdot 10^{8} & 1.56\cdot 10^{36} 
      & 1.80\cdot 10^{144} & 1.85\cdot 10^{500} \\ \hline
  \end{array}
\]
And if Theorem~\ref{theorem:buildarrows}(d) gives a lower bound
for~\eqref{eqn:pd11psqrt2} that is larger than~$\sqrt{2}$, then the
lower bounds for~$\nu(d)$ and~$d$ in terms of~$p_{\min}(d)$ will be
even larger.
\end{remark}

\begin{remark}
\label{remark:precbds}
The formulas~\eqref{eqn:pdrppqIn} and~\eqref{eqn:pdrppqIn2} derived
during the course of proving Theorem~\ref{theorem:buildarrows}(d)
impose stringent conditions on the allowable values of~$d$.  We used
these formulas to derive a general lower bound, but when analyzing a
specific~EDS, it is probably best to use them directly.  We also note,
although we will not prove, that~\eqref{eqn:pdrppqIn} is true even
if~$d$ is divisible by primes for which~$P$ has singular reduction.
Similarly, the following version of~\eqref{eqn:pdrppqIn2} is true in
general:
\[
   \prod_{p\mid d} \frac{\#\Neron(\FF_p)}{p}
  \ge  \Bigl(\prod_{q\mid d} q^{\In(q)-1}\Bigr)
     \Bigl(\prod_{p\mid d} M_p\Bigr),
\]
where $\Neron$ is the N\'eron model of~$E$.  Note that if~$E$ has bad
reduction, then $\#\Neron(\FF_p)=c_p\#E_\ns(\FF_p)$, where~$c_p$ is
the number of components in the special fiber above~$p$. In
particular,~$c_p\le4$ unless the reduction is split multiplicative, in
which case $c_p=\ord_p(\Disc(E))$.
\end{remark}

\begin{example}
\label{example:N37a}
Continuing with the EDS associated to the elliptic curve and point
from Example~\ref{example:N37curve}, we have
\[
  \Disc(E)=37,\quad
  \#E(\FF_2) = 5,\quad
  \#E(\FF_3) = 7,\quad
  \#E(\FF_5) = 8.
\]
In particular,~$E$ has multiplicative reduction at~$37$ and good
reduction elsewhere, the point~$P$ is in~$E_\ns(\FF_{37})$, and
$\#E(\FF_p)\ne2p$ for all primes~$p$.  Further, since $D_5=2$ and
$D_{10}=4$, we see that Lemma~\ref{lemma:formalgp}(b) is true even
for~$p=2$ and all values of~$n$ and~$k$, so we can treat~$2$ as we do
all other primes.
\par
We claim that for all primes~$p$,
\[
  p\mid D_n\quad\text{or}\quad p\in\Aliq(\EDS)
  \quad\Longleftrightarrow\quad (n\to np)\in\Arrow(\EDS).
\]
The implication $\Rightarrow$ follows directly from
Theorem~\ref{theorem:buildarrows}(a,c).  Conversely, if~$(n\to
np)\in\Arrow(\EDS)$, then either~$p\mid D_n$, or else
Theorem~\ref{theorem:buildarrows}(c) tells us that~$p\in\Aliq(\EDS)$.
We thus have a precise description of the arrows of prime weight.
\par
Theorem~\ref{theorem:buildarrows}(d) says that arrows $(n\to nd)$ of
composite weight with $d\notin\AliqGen(\EDS)$ have~$d$ values that are
either divisible by small primes or are huge. Further, examining the
proof of Theorem~\ref{theorem:buildarrows}(d) shows that the prime
divisors of such~$d$ must satisfy some fairly stringent conditions. We
suspect that for this example there are no such arrows, i.e., 
\[
  \left(\begin{tabular}{@{}c@{}}
    $d$ is prime\\ and $d\mid D_n$\\
  \end{tabular}\right)
  \quad\text{or}\quad d\in\AliqGen(\EDS)
  \quad\stackrel{?}{\Longleftrightarrow}\quad (n\to nd)\in\Arrow(\EDS).
\]
\end{example}

\begin{example}
\label{example:nonstandardarrow1}
The following example shows that ``non-standard'' arrows exist
(cf.\ Remark~\ref{remark:nonstandardarrows}).  Let~$\EDS$ be the EDS
associated to
\[
  E: y^2 + 2xy + y = x^3 + x^2 + 7x + 4
  \quad\text{and}\quad
  P = (4,7).
\]
The curve~$E$ is nonsingular at~$2$,~$3$, and~$5$, and
\[
  \#E(\FF_2) = 3,\qquad
  \#E(\FF_3) = 5,\qquad
  \#E(\FF_5) = 6.
\]
Further, the point~$P$ has exact order~$6$ in~$E(\FF_5)$.  Thus
$r_2=3$, $r_3=5$, and $r_5=6$, so 
\begin{equation}
  \label{eqn:23561015notinS}
  2,3,5,6,10,15\notin\Scal(\EDS) \qquad\text{and}\qquad
  1,30\in\Scal(\EDS).
\end{equation}
Alternatively, we can verify~\eqref{eqn:23561015notinS} directly by
explicitly computing the relevant terms of~$\EDS$,
\begin{align*}
  D_{1} \bmod 1 &= 0, &
  D_{2} \bmod 2 &= 1, &
  D_{3} \bmod 3 &= 2, &
  D_{5} \bmod 5 &= 4, \\
  D_{6} \bmod 6 &= 4, &
  D_{10} \bmod 10 &= 3, &
  D_{15} \bmod 15 &= 3, &
  D_{30} \bmod 30 &= 0.
\end{align*}
It follows from the definition of the directed graph~$\Scal(\EDS)$ that
$(1\to30)\in\Arrow(\EDS)$.  However, since $r_5=6$ is not prime, we
have $30\notin\AliqGen(\EDS)$.  Thus the arrow $(1\to30)$ is not
predicted by Theorem~\ref{theorem:buildarrows}(d-i). This does not
contradict the theorem, of course, since 
\[
  \frac{\#E(\FF_2)}{2}\cdot\frac{\#E(\FF_3)}{3}\cdot\frac{\#E(\FF_5)}{5}
  = \frac{3}{2}\cdot\frac{5}{3}\cdot\frac{6}{5} = 3,
\]
so condition~\eqref{eqn:minpigt212l12x} is satisfied and we are in the
situation of Theorem~\ref{theorem:buildarrows}(d-ii).
\end{example}

\begin{remark}
\label{remark:CRTarrows}
Generalizing Example~\ref{example:nonstandardarrow1}, we sketch how 
to construct EDS having non-standard arrows with arbitrarily large
values of~$d$.  The proof of Theorem~\ref{theorem:buildarrows}(d)
suggests the method.  We start with primes~$p_1,\ldots,p_N$ and
integers $n_1,\ldots,n_N$ and $k_1,\ldots,k_N$ satisfying
\[
  \left| p_i + 1 - k_in_i \right| < 2\sqrt{p_i}.
\]
Our goal is to find an elliptic curve~$E/\QQ$ and point~$P\in E(\QQ)$
such that $\#E(\FF_{p_i})=k_in_i$ and $r_{p_i}=n_i$ for all $1\le i\le N$.
\par
A theorem of Deuring \cite{Deuring} says that there exists an
elliptic curve $E_i/\FF_{p_i}$ satisfying
\[
  \#E_i(\FF_{p_i})=k_in_i,
\]
and a result of R\"uck~\cite[Theorem 3]{Ruck} says that we can choose~$E_i$
so that the group structure of $E_i(\FF_{p_i})$ ensures the existence
of a point~$P_i \in E_i(\FF_p)$ of order~$n_i$.
Making a change of coordinates, we may assume that $P_i = (0,0)$.
\par
Next we apply the Chinese remainder theorem to the coefficients of the
Weierstrass equations of $E_1,\ldots,E_n$. This gives an elliptic
curve $E/\QQ$ with $(0,0)\in E(\QQ)$ that satisfies
\[
  E \bmod p_i \cong E_i, \quad 1 \le i \le N.
\]
If the Weierstrass equation for~$E$ is not globally minimal, then we
can change coordinates to make it minimal without affecting the
reduction at~$p_1,\ldots,p_N$, since they are primes of good
reduction. For simplicity, we will assume that some~$n_i$ is divisible
by a prime greater than~$7$, since then Mazur's Theorem
\cite[VIII.7.5]{AEC} ensures that~$(0,0)$ is not a torsion point.  We
may thus associate to~$E$ and~$P$ an elliptic divisibility
sequence $\EDS = (D_n)_{n\ge1}$ satisfying
\[
   r_{p_i} = n_i \quad\text{for all $1 \le i \le N$.}
\]
\par
Finally, we observe that arbitrarily large non-standard arrows can be
constructed in this way. We begin with any prime $p_1$, we let
\[
  p_1, p_2, \ldots, p_N
\]
be a list of consecutive primes, and we set
\[
  d = p_1^2 p_2 p_3 \cdots p_N.
\]
We then find a curve and point whose associated EDS satisfies
\[
  r_{p_{i}} = p_{i+1}, \quad 1 \le i \le N-1, \quad\text{and}\quad r_{p_{N}} = p_1^2.
\]
If the list of primes is taken to be long enough, then the final
condition~$r_{p_N} = p_1^2$ is allowed by Hasse's bound, and we can
proceed as in the description above to find a sequence~$\EDS =
(D_n)_{n \ge 1}$ with  $(1 \to d)\in\Arrow(\EDS)$.
\end{remark}

\begin{example}
We use the method described in Remark~\ref{remark:CRTarrows} to
construct a non-standard arrow $(1\to d)$ for the moderately large
integer
\[
d = 5^2 \cdot 7\cdot 11 \cdot 17 = 32725.
\]
We want to construct an elliptic curve $E/\QQ$ and point $P \in
E(\QQ)$ satisfying
\begin{equation}
\label{eqn:rd}
  r_{5} = 7, \quad
  r_{7} = 11, \quad
  r_{11} = 17, \quad\text{and}\quad
  r_{17} = 25.
\end{equation}
Then the associated sequence $\EDS = (D_n)_{n\ge1}$ will have $(1 \to
d) \in \Arrow(\EDS)$, according to Proposition~\ref{proposition:amiq}.
\par
To do this, we first found elliptic curves $E_5/\FF_5$, $E_7/\FF_7$,
$E_{11}/\FF_{11}$ and $E_{17}/\FF_{17}$ satisfying
\[
  \# E_5(\FF_{5}) = 7, \quad
  \# E_7(\FF_{7}) = 11, \quad
  \# E_{11}(\FF_{11}) = 17,\quad
  \# E_{17}(\FF_{17}) = 25.
\]
This is possible because the Hasse bound is satisfied in each
instance.  We then used the Chinese remainder theorem to find an
elliptic curve $E$ with minimal Weierstrass equation
\[
  y^2 + y = x^3 + x^2 - 1291874622406186x + 17872226251073822113702,
\]
and point
\[
  P = (20751503, 1073344).
\]
(We've moved~$P$ away from~$(0,0)$ to make the numbers a bit smaller.)
The associated sequence $\EDS = (D_n)_{n \ge 1}$ begins
\begin{multline*}
  1,\enspace 2146689,\enspace 286883381041833542301,\\
  60768120452650698495048133538894517, \ldots
\end{multline*}
By construction, $(1 \to 32725) \in \Arrow(\EDS)$.  Of course, the
$32725^{\text{th}}$ term is too large to print, but the claim can be verified
by computation modulo $32725$.

For this example, we can verify equation \eqref{eqn:minpigt212l12x} in
Theorem~\ref{theorem:buildarrows}(d), which states
\begin{equation}
  \prod_{p\mid d} \left(1+\frac{1}{\sqrt{p}}\right)^2
  \ge p_0^t.
\end{equation}
In our case, $p_0=5$, $t=1$, and the left-hand side exceeds $10$.
\end{example}

\begin{remark}
\label{remark:singularEDS}
In the definition of EDS, the elliptic curve may be replaced with a
singular cubic curve as long as~$P$ is a non-singular point, since
$E_\ns(\QQ)$ is a group.  More precisely,~$E_\ns(\QQ)$ is either the
additive group~$\QQ^+$, the multiplicative group~$\QQ^*$, or a
subgroup of a quadratic twist of the multiplicative
group; see~\cite[III.2.5, Exercise~3.5]{AEC}.  Thus EDS on singular
elliptic curves are closely related to Lucas sequences.

For example, consider the nodal singular cubic curve and point
\[
C: y^2 + 3xy + 3y = x^3 + 2x^2 + x,
  \quad\text{and}\quad
P = (0,0).
\]
The associated EDS,
\[
 \EDS : 1, 3, 8, 21, 55, 144, 377, 987, 2584, 6765, \ldots
\]
consists of the even-indexed Fibonacci numbers. This is exactly the
Lucas sequence generated by
\[
  L_{n+2} = 3L_{n+1} - L_n,\qquad L_0=0,\quad L_1=1.
\]
The index divisibility set of~$\EDS$ is
\[
  \Scal(\EDS) = \{ 1, 5, 6, 12, 18, 24, 25, 30, 36, 48, 54, 55, 60,
  72, 84, \ldots \}.
\]
In the notation of Smyth's Theorem~\ref{theorem:smyth}, we have
\[
  a=3,\quad b=1,\quad \D=5,\quad\text{and}\quad
 \Bcal_{3,1} = \{ 1 \to 6 \}.
\]
In the language of our paper, $5,6 \in \AliqGen(\EDS)$, 
since
\[
  r_2=3,\quad r_3=2,\quad\text{and}\quad r_5=5.
\]
Thus~$(2,3)$ and~$(5)$ are generalized aliquot cycles.
Notice that the curve $C$ reduces modulo $p$ to a curve having $p$, $p-1$
or $p+1$ non-singular points according as $p$ ramifies, splits, or is
inert in~$\QQ(\sqrt{5})$. 
\par
In general, our Theorem~\ref{theorem:buildarrows} and
Smyth's Theorem~\ref{theorem:smyth} can probably be combined
into a general theorem on (possibly singular) cubic curves.  Notice
that Smyth's set $\Bcal_{a,b}$ may include non-standard arrows in the
case of the multiplicative group, although the analysis is simpler
because $\#C_\ns(\FF_p) \in \{p,p+1,p-1\}$. The primes~$p$ dividing
\text{$\D=a^2-4b$} are the primes for which the group underlying the
Lucas sequence reduces to the additive group~$\FF_p^+$. They are thus
analogous to the primes of additive reduction whose arrows $(n\to np)$
are described in Theorem~\ref{theorem:buildarrows}(a,c).  We also note
that in the multiplicative group case we never have $r_2=4$, so we are
always in the $2$-regular setting.
\end{remark}

\section{Elliptic aliquot cycles}
\label{section:aliquotexamples}
Let $\EDS$ be an EDS with associated elliptic curve and point $(E,P)$,
and let $(p,q)\in\Aliq(\EDS)$ be an amicable pair for $\EDS$. Then the
point~$P$ has order~$q$ modulo~$p$, and~$P$ has order~$p$
modulo~$q$. This implies that
\[
  q\mid\#E(\FF_p)\qquad\text{and}\qquad p\mid\#E(\FF_q).
\]
Conversely, if we are given~$\EDS$ and~$(E,P)$, and if~$p$ and~$q$ are
distinct primes of good reduction satisfying
\begin{equation}
  \label{eqn:amicableE}
  \#E(\FF_p)=q\qquad\text{and}\qquad\#E(\FF_q)=p,
\end{equation}
then~$(p,q)$ is automatically an amicable pair for~$\EDS$.
\par
We note that the conditions~\eqref{eqn:amicableE} do not refer to the
point~$P$. This leads to the following definitions.

\begin{definition}
Let $E/\QQ$ be an elliptic curve.  
An \emph{aliquot cycle of length~$\ell$}
for~$E/\QQ$ is a sequence $(p_1,p_2,\ldots,p_\ell)$ of distinct primes
such that~$E$ has good reduction at every~$p_i$ and
\[
  \#E(\FF_{p_1})=p_2,\semiquad   \#E(\FF_{p_2})=p_3,\semiquad\ldots,\semiquad
  \#E(\FF_{p_{\ell-1}})=p_\ell,\semiquad   \#E(\FF_{p_\ell})=p_1.
\]
An \emph{amicable pair} for~$E/\QQ$ is an aliquot cycle of length~$2$.
\end{definition}

\begin{remark}
The distribution of amicable pairs and aliquot cycles on elliptic
curves is studied in~\cite{SilvStangEAS}. In particular, it turns out
that elliptic curves with complex multiplication behave quite
differently from curves without CM.  For the convenience of
the reader, we briefly summarize
some of the material in~\cite{SilvStangEAS}.
\begin{parts}
\item[\textbullet]
If~$E(\QQ)$ contains a non-trivial torsion point, then~$E$ has
(essentially) no aliquot cycles. This is clear since
$E(\QQ)_\tors\hookrightarrow E(\FF_p)$ for all primes
\text{$p\notdivide2\Disc_{E/\QQ}$}; cf.\ \cite[Remark~5]{SilvStangEAS}.
\item[\textbullet]
For any~$\ell$, there exists an elliptic curve $E/\QQ$ that has an
aliquot cycle of length~$\ell$. More generally, for any
$\ell_1,\ldots,\ell_s$ there exists an elliptic curve having disjoint
aliquot cycles of
length~$\ell_1,\dots,\ell_s$~\cite[Theorem~13]{SilvStangEAS}.
\item[\textbullet]
Let $E/\QQ$ be an elliptic curve with complex multiplication
and $j(E)\ne0$. Then $E$ has no  aliquot cycles of length $\ell\ge3$
composed of primes $p\ge5$~\cite[Corollary~16]{SilvStangEAS}.
\item[\textbullet]
Let $E/\QQ$ be an elliptic curve with $j(E)=0$. Then $E$ has 
no aliquot cycles of length~$3$ composed of primes
$p\ge11$~\cite[Proposition~48]{SilvStangEAS}.
\item[\textbullet]
\emph{Conjecture}: Assume that there are infinitely many primes~$p$ such
that $\#E(\FF_p)$ is prime. If~$E$ does not have~CM, then
\[
  \#\{\text{aliquot cycles $(p_1,\ldots,p_\ell)$ with $p_i\le X$}\}
  \gg\ll \frac{\sqrt{X}}{(\log X)^\ell}.
\]
If~$E$ has~CM, then there is a constant~$C_E>0$ such that
\[
  \#\{\text{amicable pairs $(p,q)$ with $p,q\le X$}\}
  \sim C_E\frac{X}{(\log X)^2}.
\]
\end{parts}
\end{remark}

The next proposition shows that aliquot cycles for an elliptic
divisibility sequence are closely related to aliquot cycles on
the associated elliptic curve.

\begin{proposition} 
\label{proposition:aliqisaliq}
Let~$\EDS$ be a minimal EDS, and let~$(E,P)$ be the associated elliptic
curve~$E/\QQ$ and point~$P\in E(\QQ)$.  
\begin{parts}
\Part{(a)}
Let $(p_1,\ldots,p_\ell)$ be an aliquot cycle for~$E/\QQ$ such that
$p_i\notdivide D_1$ for all~$i$.  Then $(p_1,\ldots,p_\ell)$ is an
aliquot cycle for~$\EDS$.
\Part{(b)}
Let $(p_1,\ldots,p_\ell)$ be an aliquot cycle for~$\EDS$.
Then
\begin{equation}
  \label{eqn:ppmdfEFp2}
  \prod_{i=1}^\ell \frac{\#E(\FF_{p_i})}{p_i} < 2
  \implies
  \left(\begin{tabular}{@{}l@{}}
    $(p_1,\ldots,p_\ell)$ is an\\ aliquot cycle for~$E$\\
    \end{tabular}
  \right).
\end{equation}
In particular, 
\begin{equation}
  \label{eqn:ppmdfEFp2x}
  \min_{1\le i\le \ell} p_i > \frac{1}{(2^{1/2\ell}-1)^{2}}
  \implies
  \left(\begin{tabular}{@{}l@{}}
    $(p_1,\ldots,p_\ell)$ is an\\ aliquot cycle for~$E$\\
    \end{tabular}
  \right),
\end{equation}
cf.\ Theorem~\textup{\ref{theorem:buildarrows}(d)}.  
\end{parts}
\end{proposition}
\begin{proof}
(a)
If $(p_1,p_2,\ldots,p_\ell)$ is an aliquot cycle for~$E/\QQ$, then for
all~$i$ we know that $\#E(\FF_{p_i})=p_{i+1}$ is prime.
Since~$p_{i+1}\notdivide D_1$, the order of the point~$P$ in
$E(\FF_{p_i})$ must equal~$p_{i+1}$.  Therefore
$r_{p_i}(\EDS)=p_{i+1}$, so the cycle is aliquot for~$\EDS$.
\par\noindent(b)\enspace
The proof is similar to the proof of
Theorem~\ref{theorem:buildarrows}(d).
We are given that
$r_{p_i}(\EDS)=p_{i+1}$ for all~$i$, or equivalently, the point~$P$
has order~$p_{i+1}$ in the group~$E(\FF_{p_i})$. Thus for every $1\le
i\le \ell$ we have
\[
  \#E(F_{p_i}) = p_{i+1}M_{p_i}
  \quad\text{for some $M_{p_i}\ge1$.}
\]
Multiplying for $1\le i\le\ell$ and dividing by~${p_1\cdots p_\ell}$
yields
\[
  \prod_{i=1}^\ell \frac{\#E(\FF_{p_i})}{p_i} 
  = \prod_{i=1}^\ell M_i.
\]
Thus the assumption that $\prod_{i}\#E(\FF_{p_i})/p_i<2$ implies
that~$M_i=1$ for every~$i$, so $(p_1,\ldots,p_\ell)$ is an aliquot
cycle for~$E$. This proves~\eqref{eqn:ppmdfEFp2}.
\par
To prove~\eqref{eqn:ppmdfEFp2x}, we use the Hasse--Weil bound
${\#E(\FF_p)}\le(\sqrt{p}+1)^2$ to obtain
\[
  \prod_{i=1}^\ell \frac{\#E(\FF_{p_i})}{p_i} 
  \le \prod_{i=1}^\ell \left(1 + \frac{1}{\sqrt{p_i}}\right)^2
  \le \left(1 + \frac{1}{\min_i\sqrt{p_i}}\right)^{2\ell}.
\]
Now a little bit of algebra, combined with~\eqref{eqn:ppmdfEFp2}
yields~\eqref{eqn:ppmdfEFp2x}.
\end{proof}

\section{Miscellaneous Remarks}
\label{section:miscremarks}

We conclude with two brief remarks.

\begin{remark}
Recall that a sequence~$\Aseq=(A_n)_{n\ge1}$ is called a
\emph{divisibility sequence} if
\[
  m\mid n \Longrightarrow A_m \mid A_n.
\]
Examples of divisibility sequences include Lucas sequences of the
first kind, the odd terms of Lucas sequences of the second kind, and
elliptic divisibility sequences. We observe that if~$\Aseq$ is a
divisibility sequence, then 
\[
  n\in\Scal(\Aseq)\quad\text{and}\quad d\mid D_n
    \quad\text{and}\quad \gcd(n,d)=1
  \implies nd\in\Scal(\Aseq).
\]
In particular, there is a sequence of arrows in~$\Arrow(\Aseq)$ 
satisfying
\[
  n \to \cdots \to nd.
\]
This is one way in which the index divisibility graph of divisibility
sequences exhibits a structure not found for arbitrary sequences.
It might be interesting to see if there are any other general statements
that one can make about the index divisibility graph of general
divisibility sequences.
\end{remark}

\begin{remark}
\label{remark:wardedsdef}
A classical alternative definition of an elliptic divisibility sequence
is a sequence of integers $\WardEDS=(W_n)_{n\ge1}$ 
defined by four initial terms~$(W_1,W_2,W_3,W_4)$ and satisfying
the recursion
\[
  W_{n+m}W_{n-m}W_r^2 = W_{n+r}W_{n-r}W_m^2 - W_{m+r}W_{m-r}W_n^2
  \text{ for all } n > m > r.
\]
One can show that if the sequence is normalized by $W_1=1$ and
\text{$W_2\mid W_4$}, then every term is an integer.
Ward~\cite{MR0027286,MR0023275} was the first to study the arithmetic
properties of these sequences. Subject to some non-degeneracy
conditions, he showed that there is an elliptic curve~$E/\QQ$ given by
a Weierstrass equation and a point~$P\in E(\QQ)$ such that
\[
  W_n = \psi_n(P),
\]
where~$\psi_n$ is the~$n$'th division polynomial
for~$E$~\cite[Exercise~3.7]{AEC}.  (See \cite{MR0027286}
or~\cite[Appendix~A]{MR2226354} for explicit formulas for~$E$ and~$P$
in terms of the initial terms of the~EDS.)  In particular, if
$\EDS=(D_n)_{n\ge1}$ is the EDS associated to~$(E,P)$, then~$D_n\mid W_n$ for
all~$n\ge1$. Thus
\begin{equation}
  \label{eqn:ndivDndivWn}
  n\mid D_n \implies n\mid W_n,
\end{equation}
so index divisibility for~$\EDS$ is a stronger condition than it is
for~$\WardEDS$. Further, one can show that
\[
  \ord_p(D_n) = \ord_p(W_n)
\]
for all primes~$p$ at which the Weierstrass equation has good
reduction, so the implication~\eqref{eqn:ndivDndivWn} can be reversed
if we ignore primes of bad reduction.  This shows that the
divisibility properties of~$\EDS$ and~$\WardEDS$ are closely
related. We have chosen in this paper to concentrate on the former.
\end{remark}

\begin{acknowledgement}
The research in this note was performed while the first author was a
long-term visiting researcher at Microsoft Research New England and
included a short visit by the second author. Both authors thank MSR
for its hospitality during their visits.
\end{acknowledgement}





\end{document}